\documentstyle[amssymb,amsfonts,12pt]{amsart}

\newtheorem{theorem}{Theorem}[section]
\newtheorem{lemma}[theorem]{Lemma}

\newtheorem{proposition}[theorem]{Proposition}
\newtheorem{definition}[theorem]{Definition}

\newtheorem{question}[theorem]{Question}\newtheorem{conjecture}[theorem]{Conjecture}

\theoremstyle{remark}

\newtheorem*{ack*}{Acknowledgment}

\def\QSet{\mbox{\rm\kern.24em
\vrule width.03em height1.48ex depth-.051ex \kern-.26em Q}}

\def\x{{\bf x}}
\def\T{{\mathbb T}}

\def\R{{\mathbb R}}
\def\E{{\mathbb E}}
\def\N{{\mathbb N}}
\def\C{{\mathbb C}}

\def\Z{{\mathbb Z}}
\def\P{{\mathcal P}}
\def\H{{\mathcal T}}
\def\W{{\mathcal L}}
\def\A{{\mathcal N}}
\def\dist{{\operatorname{dist}}}
\def\supp{{\operatorname{supp}}}
\def\bas{\begin{align*}}
\def\eas{\end{align*}}
\def\bi{\begin{itemize}}
\def\ei{\end{itemize}}
\newenvironment{proof}{\noindent {\bf Proof} }{\endprf\par}
\def \endprf{\hfill  {\vrule height6pt width6pt depth0pt}\medskip}
\def\emph#1{{\it #1}}

\begin{document}

\title[Incidence theory and restriction estimates]{Incidence theory and discrete restriction estimates}

\author{Ciprian Demeter}
\address{Department of Mathematics, Indiana University, 831 East 3rd St., Bloomington IN 47405}
\email{demeterc@@indiana.edu}

\thanks{The author is partially supported by  the NSF Grant DMS-1161752}
\begin{abstract}
We investigate the interplay between the discrete restriction phenomenon and incidence theory. Two angles are explored. One is a refinement of the machinery developed by Thomas Wolff, which when combined with a recent subcritical estimate of Bourgain leads to new Strichartz estimates for irrational tori. The other one connects various additive energies with the Szemer\'edi-Trotter Theorem.  The combination of the two approaches  recovers  the best known Strichartz estimates for the classical torus in dimensions $n\ge 3$, without any appeal to number theory.
\end{abstract}
\maketitle

\section{The discrete restriction phenomenon}
\bigskip

We start by recalling the Stein-Tomas Theorem  for  the sphere $S^{n-1}$ and the truncated paraboloid
$$P^{n-1}:=\{(\xi_1,\ldots,\xi_{n-1},\xi_1^2+\ldots+\xi_{n-1}^2)\in\R^n:\;|\xi_i|\le 1/2\}.$$
Throughout the whole paper we will implicitly assume $n\ge 2$.
\begin{theorem}
Let $S$ be either $S^{n-1}$ or $P^{n-1}$ and let $d\sigma$ denote its natural surface measure. Then for $p\ge \frac{2(n+1)}{n-1}$ and $f\in L^2(S,d\sigma)$ we have
$$\|\widehat{fd\sigma}\|_{L^p(\R^n)}\lesssim\|f\|_{L^2(S)}.$$
\end{theorem}
We will use the notation $e(a)=e^{2\pi i a}$. For fixed $p\ge \frac{2(n+1)}{n-1}$, it is an easy exercise to see that this theorem is equivalent with the statement that
$$(\frac{1}{|B_R|}\int_{B_R}|\sum_{\xi\in \Lambda}a_\xi e(\xi \cdot\x)|^{p})^{1/p}\lesssim \delta^{\frac{n}{2p}-\frac{n-1}{4}}\|a_\xi\|_{l^2(\Lambda)},$$
for each $0\le 1<\delta$, each $a_\xi\in\C$, each ball $B_R\subset \R^n$ of radius $R\sim\delta^{-1/2}$  and each $\delta^{1/2}$ separated set $\Lambda\subset S$.  Thus, the Stein-Tomas  Theorem measures the average $L^p$ oscillations of exponential sums at spatial scale equal to the inverse of the separation of the frequencies. It will be good to keep in mind that for each $R\gtrsim \delta^{-1/2}$
\begin{equation}
\label{EE43}
\|\sum_{\xi\in \Lambda}a_\xi e(\xi \cdot\x)\|_{L^2(B_R)}\sim |B_R|^{1/2}\|a_\xi\|_2,
\end{equation}
as can be seen using Plancherel's Theorem.

It has been observed that stronger cancellations occur at the larger scale $R\gtrsim \delta^{-1}$. Based on the evidence so far it seems very reasonable to conjecture the following.
\begin{conjecture}[General discrete restriction]
\label{c1}
Let $\Lambda\subset S$ be a  $\delta^{1/2}$- separated set and let $R\gtrsim \delta^{-1}$. Then for each $\epsilon>0$
\begin{equation}
\label{EE42}
(\frac{1}{|B_R|}\int_{B_R}|\sum_{\xi\in \Lambda}a_\xi e(\xi \cdot\x)|^{p})^{1/p}\lesssim_{\epsilon} \delta^{\frac{n+1}{2p}-\frac{n-1}{4}-\epsilon}\|a_\xi\|_2
\end{equation}
if $p\ge \frac{2(n+1)}{n-1}$, and
\begin{equation}
\label{EE44}
(\frac{1}{|B_R|}\int_{B_R}|\sum_{\xi\in \Lambda}a_\xi e(\xi \cdot\x)|^{p})^{1/p}\lesssim_{\epsilon} \delta^{-\epsilon}\|a_\xi\|_2,
\end{equation}
if $1\le p<  \frac{2(n+1)}{n-1}$.
\end{conjecture}

Throughout the whole paper, the implicit constants hidden inside the notation  $\lesssim$ will in general  depend on $p$ and $n$, but we will not record this dependence.

In addition to $p=\infty$,  the only trivial range in Conjecture \ref{c1} is $1\le p\le 2$, in light of \eqref{EE43} and H\"older.
Note that both the supercritical estimate \eqref{EE42} and the subcritical one \eqref{EE44} would follow  if the conjectured estimate held true for the critical index $p=\frac{2(n+1)}{n-1}$. One can see this via H\"older and the estimate for $p=\infty$.

It is reasonable to hope that in the subcritical regime one may be able to replace  $\delta^{-\epsilon}$  by a constant $C_{p,n}$ independent of $\delta$. This is indeed known (and fairly immediate, via a geometric argument) when $n=2$ and $p\le 4$, but seems to be in general an extremely difficult question. To the author's knowledge, no other examples of $2<p<\frac{2(n+1)}{n-1}$ are known for when this holds.
\bigskip

A very robust machinery for proving discrete restriction estimates is the one relying on  what we will refer to in this paper as {\em decouplings}. In short, the contribution from pieces of the Fourier transform localized on pairwise disjoint caps on $S$ will get decoupled via analytic, geometric and topological mechanisms.

Let $\A_\delta$ be the $\delta$ neighborhood of $P^{n-1}$ and let $\P_\delta$ be a finitely overlapping cover of  $\A_\delta$ with curved regions $\theta$ of the form
\begin{equation}
\label{EE14}
\theta=\{(\xi_1,\ldots,\xi_{n-1},\eta+\xi_1^2+\ldots+\xi_{n-1}^2):\;(\xi_1,\ldots,\xi_{n-1})\in C_\theta,\;|\eta|\le 2\delta\},
\end{equation}
where $C_\theta$ runs over all cubes   $c+[-\frac{\delta^{1/2}}{2},\frac{\delta^{1/2}}{2}]^{n-1}$ with $c\in \frac{\delta^{1/2}}{2}\Z^{n-1}\cap [-1/2,1/2]^{n-1}$. Note that each $\theta$ sits inside a $\approx\delta^{1/2}\times \ldots\delta^{1/2}\times \delta$ rectangular box. A similar decomposition exists for the sphere and we will use the same notation $\P_\delta$ for it. We will denote by $f_\theta$ the Fourier restriction of $f$ to $\theta$. It seems reasonable to conjecture the following.
\begin{conjecture}[$l^2$ weak decoupling conjecture]
\label{c2}
Let $S$ be either $P^{n-1}$ or $S^{n-1}$ and assume $\supp(\hat f)\subset \A_\delta$.
Then for  $p\ge \frac{2(n+1)}{n-1}$ and $\epsilon>0$ we have
$$
\|f\|_p\lesssim_\epsilon \delta^{-\frac{n-1}{4}+\frac{n+1}{2p}-\epsilon}(\sum_{\theta\in \P_\delta}\|f_\theta\|_p^2)^{1/2},
$$
while if $2\le p\le \frac{2(n+1)}{n-1}$ we have
$$
\|f\|_p\lesssim_\epsilon \delta^{-\epsilon}(\sum_{\theta\in \P_\delta}\|f_\theta\|_p^2)^{1/2}.
$$
\end{conjecture}

We point out that Wolff \cite{TWol} has initiated the study of $l^p$ weak decouplings, $p>2$. His work provides a lot of the inspiration for our paper. We mention briefly that there is a stronger form of decoupling, sometimes referred to as {\em square function estimate}, which predicts that
\begin{equation}
\label{kjfehrfyryufjhfuygyrufdhcblkskdiopweru}
\|f\|_p\lesssim_\epsilon \delta^{-\epsilon}\|(\sum_{\theta\in \P_\delta}|f_\theta|^2)^{1/2}\|_p,
\end{equation}
for $2\le p\le \frac{2n}{n-1}$. Minkowski's inequality shows that \eqref{kjfehrfyryufjhfuygyrufdhcblkskdiopweru} is indeed stronger than Conjecture \ref{c2} in the range $2\le p\le \frac{2n}{n-1}$. This  is also confirmed by the lack of any results for \eqref{kjfehrfyryufjhfuygyrufdhcblkskdiopweru} when $n\ge 3$.

In analogy to Conjecture \ref{c1}, the result for the critical index $p=\frac{2(n+1)}{n-1}$ in Conjecture \ref{c2} would imply the result for all other values of $p$. Conjecture \ref{c2} is false for $p<2$. This can easily be seen by testing it with functions of the form $g_\theta(x)=f_\theta(x+c_\theta)$, where the numbers $c_\theta$ are very far apart from each other.

It has been observed in \cite{Bo2} that Conjecture \ref{c2} for a given $p$ implies Conjecture \ref{c1} for the same $p$. Here is a sketch of the argument. First, note that the statement
$$\|f\|_p\lesssim \delta^{c_p}(\sum_{\theta\in \P_\delta}\|f_\theta\|_p^2)^{1/2},\;\text{whenever }\supp(\hat f)\subset \A_\delta$$
easily implies that for each $f:S\to\C$ and $R\gtrsim \delta^{-1}$
\begin{equation}
\label{EE45}
(\int_{B_R}|\widehat{fd\sigma}|^p)^{1/p}\lesssim \delta^{c_p}(\sum_{\theta\in \P_\delta}\|\widehat{(f^\theta d\sigma})w_{B_{R}}\|_{L^p(\R^n)}^2)^{1/2},
\end{equation}
where  $f^\theta=f1_\theta$ is the restriction of $f$ to the $\delta^{1/2}$- cap $\theta$ on $S$ and $$w_{B_R}(x)=(1+\frac{|x-c(B_R)|}{R})^{-100}.$$
It now suffices to use $f=\sum_{\xi\in\Lambda}a_\xi \sigma(U(\xi,\tau))^{-1}1_{U(\xi,\tau)}$ in \eqref{EE45}, where $U(\xi,\tau)$ is a $\tau$- cap on $S$ centered at $\xi$, and to let $\tau\to 0$.
\bigskip

Conjectures \ref{c2} and \ref{c1} have been recently verified by Bourgain for $p=\frac{2n}{n-1}$, using induction on scales and multilinear theory.
\begin{theorem}[\cite{Bo2}]
\label{thm:Bor}
Assume $\supp(\hat f)\subset \A_\delta$.
For each $\epsilon>0$ and $p=\frac{2n}{n-1}$
\begin{equation}
\label{Snew29}
\|f\|_p\lesssim_\epsilon \delta^{-\epsilon}(\sum_{\theta\in \P_\delta}\|f_\theta\|_p^2)^{1/2}.
\end{equation}
\end{theorem}
\bigskip

We close this section by mentioning that the discrete restriction phenomenon has also been investigated in the special case when the frequency points  $\Lambda$ belong to a (re-scaled lattice). There is extra motivation in considering these problems coming from PDEs, see for example Section \ref{s2} below. The best known result for the paraboloid $$P^{n-1}(N):=\{\xi:=(\xi_1,\ldots,\xi_n)\in\Z^n:\;\xi_n=\xi_1^2+\ldots+\xi_{n-1}^2,\;|\xi_1|,\dots,|\xi_{n-1}|\le N\}$$ is due to Bourgain \cite{Bo3}, \cite{Bo2}.
\begin{theorem}[Discrete restriction: the lattice case (paraboloid)]

Let $n\ge 4$. For each $a_\xi\in\C$ and each $\epsilon>0$ we have
\begin{equation}
\label{EE47}
\|\sum_{\xi\in P^{n-1}(N)}a_\xi e(\xi\cdot \x)\|_{L^p(\T^n)}\lesssim_\epsilon N^{\frac{n-1}{2}-\frac{n+1}{p}+\epsilon}\|a_\xi\|_{l^2},
\end{equation}
for  $p\ge \frac{2(n+2)}{n-1}$ and
\begin{equation}
\label{EE46}
\|\sum_{\xi\in P^{n-1}(N)}a_\xi e(\xi\cdot \x)\|_{L^p(\T^n)}\lesssim_\epsilon N^{\epsilon}\|a_\xi\|_{l^2},
\end{equation}
for $1\le p\le \frac{2n}{n-1}$.
\end{theorem}
The proof combines the implementation of the Stein-Tomas argument via the circle method with the weak decoupling \ref{thm:Bor}.
It is conjectured that \eqref{EE46} will hold for $p\le \frac{2(n+1)}{n-1}$. This was indeed proved to be true in \cite{Bo3} in the case $n=2$ and $n=3$  by a simple argument which uses the fact that circles in the plane contain "few" lattice points. We present an alternative argument in Section \ref{s8}, that does not rely on number theory.

There is an analogue problem for the sphere, but we will not be able to say anything new about it in this paper. The reader is referred to \cite{BD1} and \cite{BD2} for the best results in this case.
\begin{ack*}
The author has benefited from helpful conversations with  Jean Bourgain and  Nets Katz.
\end{ack*}
\bigskip

\section{New results}
\label{s2}
\bigskip

In the first part of our  paper we refine  Wolff's technology from \cite{TWol} and combine this  with Theorem \ref{thm:Bor} to answer Conjecture \ref{c2} for $P^{n-1}$ in  a large part of the supercritical regime.
\begin{theorem}
\label{thmmm1}
Let $S=P^{n-1}$ and assume $\supp(\hat{f})\subset \A_\delta$.
Then for $p> \frac{2(n+2)}{n-1}$ and $\epsilon>0$
\begin{equation}
\label{Snew26}
\|f\|_p\lesssim_\epsilon \delta^{-\frac{n-1}{4}+\frac{n+1}{2p}-\epsilon}(\sum_{\theta\in \P_\delta}\|f_\theta\|_p^2)^{1/2}.
\end{equation}
\end{theorem}
As explained before, this implies
\begin{theorem}
\label{thmmmm2}
Let $\Lambda\subset P^{n-1}$ be a  $\delta^{1/2}$- separated set and let $R\gtrsim \delta^{-1}$. Then for each $\epsilon>0$
\begin{equation}
\label{EE49}
(\frac{1}{|B_R|}\int_{B_R}|\sum_{\xi\in \Lambda}a_\xi e(\xi \cdot\x)|^{p})^{1/p}\lesssim_{\epsilon} \delta^{\frac{n+1}{2p}-\frac{n-1}{4}-\epsilon}\|a_\xi\|_2
\end{equation}
if $p\ge \frac{2(n+2)}{n-1}$.
\end{theorem}
These theorems can also be shown to hold for $S^{n-1}$, the only difference is in the details on how the re-scaling is done. The details are left to the reader.
\bigskip

There are two interesting consequences of Theorem \ref{thmmmm2}. First, note that we recover \eqref{EE47} without any use of number theory.

Second, we can generalize this and obtain new Strichartz estimates for the irrational tori. More precisely, fix $\frac1{2}<\theta_1,\ldots,\theta_{n-1}<2$. For $\phi\in L^2(\T^{n-1})$  consider its Laplacian
$$\Delta \phi(x_1,\ldots,x_{n-1})=$$
$$\sum_{(\xi_1,\ldots,\xi_{n-1})\in\Z^{n-1}}(\xi_1^2\theta_1+\ldots+\xi_{n-1}^2\theta_{n-1})\hat{\phi}(\xi_1,\ldots,\xi_{n-1})e(\xi_1x_1+\ldots+\xi_{n-1}x_{n-1})$$
on the irrational torus $\prod_{i=1}^{n-1}\R/(\theta_i\T)$. Let also
$$e^{it\Delta}\phi(x_1,\ldots,x_{n-1},t)=$$$$\sum_{(\xi_1,\ldots,\xi_{n-1})\in\Z^{n-1}}\hat{\phi}(\xi_1,\ldots,\xi_{n-1})e(x_1\xi_1+\ldots+x_{n-1}\xi_{n-1}+t(\xi_1^2\theta_1+\ldots+\xi_{n-1}^2\theta_{n-1})).$$
We prove
\begin{theorem}
\label{thmmmm3}Let $\phi\in L^2(\T^{n-1})$ with $\hat{\phi}\subset [-N,N]^{n-1}$.
Then for each $\epsilon>0$, $p\ge \frac{2(n+2)}{n-1}$ and each interval $I\subset\R$ with $|I|\gtrsim 1$ we have
\begin{equation}
\label{EE52}
\|e^{it\Delta}\phi\|_{L^{p}(\T^{n-1}\times I)}\lesssim_{\epsilon} N^{\frac{n-1}{2}-\frac{n+1}{p}+\epsilon}|I|^{1/p}\|\phi\|_2,
\end{equation}
and the implicit constant does not depend on $I$, $N$ and $\theta_i$.
\end{theorem}
\begin{proof}
For $-N\le \xi_1,\ldots ,\xi_{n-1}\le N$ define $\eta_i=\frac{\theta_i^{1/2}\xi_i}{4N}$ and  $a_{\eta}=\hat{\phi}(\xi)$. A simple change of variables shows that
$$\int_{\T^{n-1}\times I}|e^{it\Delta}\phi|^p\lesssim $$$$\frac{1}{N^{n+1}}\int_{|y_1|,\ldots,|y_{n-1}|\le 8N \atop{y_n\in I_{N^2}}}|\sum_{\eta_1,\ldots,\eta_{n-1}}a_\eta e(y_1\eta_1+\ldots +y_{n-1}\eta_{n-1}+y_n(\eta_1^2+\ldots \eta_{n-1}^2))|^pdy_1\ldots dy_n,$$
where $I_{N^2}$ is an interval of length $\sim N^2|I|$.
By periodicity in the $y_1,\ldots,y_{n-1}$ variables we bound the above by
$$\frac{1}{N^{n+1}(N|I|)^{n-1}}\int_{B_{N^2|I|}}|\sum_{\eta_1,\ldots,\eta_{n-1}}a_\eta e(y_1\eta_1+\ldots +y_{n-1}\eta_{n-1}+y_n(\eta_1^2+\ldots \eta_{n-1}^2))|^pdy_1\ldots dy_n,$$
for some ball $B_{N^2|I|}$ of radius $\sim N^2|I|$. Our result will be clear if we note that the points $$(\eta_1,\ldots,\eta_{n-1},\eta_1^2+\ldots \eta_{n-1}^2)$$ are $\sim \frac1{N}$ separated on $P^{n-1}$ and then apply Theorem \ref{thmmmm2} with $R\sim N^2|I|$.
\end{proof}
\bigskip

Note that without additional assumptions on $\theta_i$, Theorem \ref{thmmmm3} is sharp up to the $\epsilon$ loss, as the lattice case  $\theta_1=\ldots=\theta_{n-1}=1$  shows.

When $n=4$ Theorem \ref{thmmmm3} answers the Problem on page 3 in \cite{Bo4}. When $n=3$ we get
\begin{equation}
\label{EE59}
\|e^{it\Delta}\phi\|_{L^{4}(\T^{2}\times I)}\lesssim_{\epsilon} N^{\frac18+\epsilon}|I|^{1/4}\|\phi\|_2,
\end{equation}
by using H\"older, the $L^3$ bound from Theorem \ref{thm:Bor} and the $L^5$ bound \eqref{EE52}. This improves on the recent result in \cite{demirb}. See also the comment at the end of this section. There are further improvements of the results from \cite{CWW} and \cite{GOW} for $n\ge 3$, that are immediate consequences of our Theorem \ref{thmmmm3}.

\vspace{0.4in}

In the second part of the paper we take a different perspective, replace the hard analysis with soft incidence theory and analyze the following weaker form of Conjecture \ref{c1}
\begin{conjecture}[Limit discrete restriction]
\label{c3}
Let $\Lambda\subset S$ be a  $\delta^{1/2}$- separated set, where $S$ is either $P^{n-1}$ or $S^{n-1}$. Then for each $\epsilon>0$ and ball $B_R$ of large enough radius $R$ we have
\begin{equation}
\label{EE55}
(\frac{1}{|B_R|}\int_{B_R}|\sum_{\xi\in \Lambda}a_\xi e(\xi \cdot\x)|^{p})^{1/p}\lesssim_{\epsilon} \delta^{\frac{n+1}{2p}-\frac{n-1}{4}-\epsilon}\|a_\xi\|_2
\end{equation}
if $p\ge \frac{2(n+1)}{n-1}$, and
\begin{equation}
\label{EE56}
(\frac{1}{|B_R|}\int_{B_R}|\sum_{\xi\in \Lambda}a_\xi e(\xi \cdot\x)|^{p})^{1/p}\lesssim_{\epsilon} \delta^{-\epsilon}\|a_\xi\|_2,
\end{equation}
if $1\le p<  \frac{2(n+1)}{n-1}$.
\end{conjecture}

When $n=3$ and $S=P^2$ we prove a stronger version of this conjecture, one where the bound does not depend on the separation between points.
\begin{theorem}
\label{thmmmm4}
Let $\Lambda\subset P^2$ be an arbitrary collection of points. We do not assume these points are $\delta^{1/2}$ separated. Then for $R$ large enough, depending only on the geometry of $\Lambda$ and on its cardinality $|\Lambda|$ we have
\begin{equation}
\label{EE61}
(\frac{1}{|B_R|}\int_{B_R}|\sum_{\xi\in \Lambda}a_\xi e(\xi \cdot\x)|^{4})^{1/4}\lesssim_{\epsilon} |\Lambda|^{\epsilon}\|a_\xi\|_2.
\end{equation}
\end{theorem}
Due to periodicity, this recovers Bourgain's result \cite{Bo3} for lattice points on the paraboloid $P^2(N)$. Our argument uses no number theory. Theorem \ref{thmmmm4} also shows that the exponent $\frac18+\epsilon$ in \eqref{EE59} can be replaced with just $\epsilon$ if $|I|$ is large enough, depending only on $N$ and the diophantine properties of $\theta_i$.

Note that \eqref{EE61} is a critical result, and thus provides further evidence that the General Discrete Restriction Conjecture \ref{c1} should be true. A tantalizing question concerns the case $n=2$ and will be described in Section \ref{s8}.

\bigskip
\section{Norms and wave packet decompositions}\label{s3}
\bigskip

Sections \ref{s3} - \ref{s7} are devoted to proving Theorem \ref{thmmm1}. Our presentation follows closely the argument in \cite{LW}. Something similar has been done in \cite{GSS} for the cone, but a weaker\footnote{While both the bilinear and the multilinear estimates are sharp, one can make stronger use of the multilinear estimate} bilinear estimate, rather than \eqref{Snew29} was used there.

We remind the reader that we use $C$ to denote various constants that are allowed to depend on $n,p,\alpha$, but never on $\delta$. Let $\delta\le 1$. To prove Theorem \ref{thmmm1} we can and  will implicitly assume $\delta$ to be a dyadic number, $\delta=2^{-k}$ for $k\in\N$.

We will use $|\cdot|$ to denote both the Lebesgue measure on $\R^n$ and the cardinality of finite sets. We write $A\lesssim B$ if $A\le CB$ for some large enough $C$ that is allowed  to depend only on $n$, $p$ and $\alpha$.
We write $A\lessapprox B$ if $A\le C\log (\frac1\delta)^CB$
for some large enough $C$ that is allowed  to depend only  on $n$, $p$ and $\alpha$.

For $2\le p\le \infty$  we define the norm
$$\|f\|_{p,\delta}=(\sum_{\theta\in\P_\delta}\|f_\theta\|_p^2)^{1/2},$$
where $f_\theta$ is the Fourier projection of $f$ onto $\theta$.
We note the following immediate consequence of H\"older's inequality
\begin{equation}
\label{EE1}
\|f\|_{p,\delta}\le \|f\|_{2,\delta}^{\frac2p}\|f\|_{\infty,\delta}^{1-\frac2p}
\end{equation}
and the fact that if $\supp(\hat{f})\subset \A_\delta$ then
$$\|f\|_{2,\delta}\sim\|f\|_2,$$

Let $\phi:\R^n\to\R$ be given by
$$\phi(x)=(1+|x|^2)^{-M},$$
for some $M$ large enough compared to $n$, whose value will become clear from the argument.

Let $\psi:\R^n\to\R$ be a rapidly decaying function such that $\psi=\eta^2$, where $\widehat{\eta}$ is supported in $B(0,\frac14)$, $|\psi|\gtrsim 1$ on $B(0,n^{1/2})$ and such that the $2^{-C}\Z^n$ translations of $\psi$ form a partition of unity for some positive integer $C$.
\bigskip

Unless specified otherwise, $f$ will always refer to a Schwartz function.
\begin{definition}
Let $N$ be a real number greater than 1. An $N$-tube $T$ is an $N^{1/2}\times\ldots N^{1/2}\times N$  rectangular parallelepiped in $\R^n$ which has dual orientation to some $\theta=\theta(T)\in \P_\delta$. We call a collection of $N$-tubes  separated if no more than $C$ tubes with a given orientation overlap.

An $N$-cube is a dyadic cube whose side length equals  $N$.
\end{definition}

Define $\phi_T=\phi\circ a_T$, where $a_T$ is the affine function mapping $T$ to the unit cube in $\R^n$ centered at the origin.
For each $N$-cube $Q$ in $\R^n$  we define $\psi_Q=\psi\circ a_Q$. We will use the fact that if $\widehat{f}$ is supported in $\A_\delta$, then $\widehat{f\psi_Q}$ is supported in $\A_{\frac\delta\sigma}$, for each $\sigma\delta^{-1}$-cube $Q$ and $\sigma\le \frac12$. We will also rely on the fact that $\inf_{x\in Q}|\psi_Q(x)|\gtrsim 1$ for each $Q$.

\begin{definition}
An $N$-function is a function $f:\R^n\to \C$ such that
$$f=\sum_{T\in \H(f)}f_T$$
where $\H(f)$ consists of finitely many separated $N$-tubes $T$ and moreover
$$|f_T|\le \phi_T,$$
$$\|f_T\|_p\sim |T|^{1/p},\;1\le p\le\infty$$
and
$$\supp (\widehat{f_T})\subset \theta(T).$$
A subfunction of $f$ will be a function of the form
$$\sum_{T\in \H'}f_T,$$
where $\H'\subset \H(f)$.

For $\theta\in \P_{1/N}$ let
$\H(f,\theta)$ denote the $N$-tubes in $\H(f)$ dual to $\theta$ .
An $N$-function is called balanced if
$|\H(f,\theta)|\le 2 |\H(f,\theta')|$
whenever $\H(f,\theta), \H(f,\theta')\not=\emptyset.$
\end{definition}
The $\|\cdot\|_{p,\delta}$ norms of $N$-functions are asymptotically determined by their  plate distribution over the sectors $\theta$.
\begin{lemma}
For each $N$-function $f$ and for  $2\le p\le \infty$
\begin{equation}
\label{EE2}
\|f\|_{p,1/N}\sim N^{\frac{n+1}{2p}}(\sum_{\theta}|\H(f,\theta)|^{\frac2p})^{1/2}.
\end{equation}
If the $N$-function is balanced then
\begin{equation}
\label{EE24}
\|f\|_{p,1/N}\sim N^{\frac{n+1}{2p}}M(f)^{\frac{1}{2}-\frac1p}|\H(f)|^{1/p},
\end{equation}
where $M(f)$ is the number of sectors $\theta$ for which $\H(f,\theta)\not=\emptyset$.
\end{lemma}
\begin{proof}
It suffices to prove \eqref{EE2} when  $\H(f)=\H(f,\theta)$ for some $\theta$. We first observe the trivial estimates $\|f\|_1\lesssim |T||\H(f)|$, $\|f\|_\infty\lesssim 1$ and $\|f\|_2\sim |T|^{1/2}|\H(f)|^{1/2}$. Applying H\"older twice we get
$$\|f\|_2^{\frac{2(p-1)}{p}}\|f\|_1^{\frac{2-p}{p}}\le\|f\|_p\le \|f\|_1^{1/p}\|f\|_\infty^{1/p'},$$
which is exactly what we want.
\end{proof}
The crucial role played by balanced $N$-functions is encoded by
\begin{lemma}
\label{lemmabalancednfunc}

(i) Each $N$-function $f$ can be written as the sum of $O(\log|\H(f)|)$ balanced $N$-functions.

(ii) For each balanced $N$-function $f$ and $2\le p\le \infty$ we have the converse of \eqref{EE1}, namely
\begin{equation}
\label{EE10}
\|f\|_{p,1/N} \sim \|f\|_{2,1/N}^{\frac2p}\|f\|_{\infty,1/N}^{1-\frac2p}.
\end{equation}
\end{lemma}
\begin{proof}
Note that $(i)$ is immediate by using dyadic ranges. Also, $(ii)$ will follow from \eqref{EE24}.
\end{proof}

In the remaining sections we will repeatedly use the fact that the contribution of $f$ to various inequalities comes from   logarithmically many $N$-functions. The basic mechanism is the following.

\begin{lemma}[Wave packet decomposition]
\label{WPD}
Assume $f$ is Fourier supported in $\A_\delta$. Then for each dyadic $0<\lambda\lesssim \|f\|_{\infty,\delta}$ there is an $N=\delta^{-1}$-function $f_\lambda$ such that
$$f=\sum_{\lambda\lesssim \|f\|_{\infty,\delta}}\lambda f_\lambda$$
and for each $2\le p<\infty$ we have
\begin{equation}
\label{EE8}
\lambda^pN^{\frac{n+1}{2}}|\H(f_\lambda)|\le \|\lambda f_\lambda\|_{p,\delta}^p\lesssim \|f\|_{p,\delta}^p.
\end{equation}
\end{lemma}
\begin{proof}
Using a partition of unity write
$$f=\sum_{\theta\in \P_\delta}\tilde{f}_\theta$$
with $\tilde{f}_\theta=f_\theta*K_\theta$ Fourier supported in $\frac9{10}\theta$ with $\|K_\theta\|_1\lesssim 1$.
Consider a windowed Fourier series expansion for each $\tilde{f}_\theta$
$$\tilde{f}_\theta=\sum_{T\in \H_\theta}\langle \tilde{f}_\theta,\varphi_T\rangle\varphi_T,$$
where $\varphi_T$ are $L^2$ normalized Schwartz functions Fourier localized in $\theta$ such that $$|T|^{1/2}|\varphi_T|\lesssim \phi_T.$$
The tubes in $\H_\theta$ are separated.
Note that by H\"older
$$|a_T:=\frac{1}{|T|^{1/2}}\langle \tilde{f}_\theta,\varphi_T\rangle|\lesssim \|\tilde{f}_\theta\|_\infty\lesssim \|f_\theta\|_\infty\le \|f\|_{\infty,\delta}.$$
It is now clear that we should take
$$f_\lambda=\sum_{\theta}\sum_{T\in\H_\theta:\;|a_T|\sim\lambda}a_T\lambda^{-1}|T|^{1/2}\varphi_T.$$

To see \eqref{EE8} note that the first inequality follows from \eqref{EE2} and the fact that $\|\cdot\|_{l^{p/2}}\le \|\cdot\|_{l^{1}}$. To derive the second inequality, it suffices to prove that for each $\theta$
$$\|\sum_{T\in\H_\theta:\;|a_T|\sim\lambda}\langle \tilde{f}_\theta,\varphi_T\rangle\varphi_T\|_p\lesssim\|f_\theta\|_p.$$
Using \eqref{EE2} and the immediate consequence of H\"older $|a_T|^p\lesssim\int|\tilde{f}_\theta|^p|T|^{-1/2}\varphi_T$ we get
$$\|\sum_{T\in\H_\theta:\;|a_T|\sim\lambda}\langle \tilde{f}_\theta,\varphi_T\rangle\varphi_T\|_p^p\lesssim \lambda^p|T||\{T\in\H_\theta:\;|a_T|\sim\lambda\}|\lesssim$$
$$\lesssim |T|\sum_{T\in\H_\theta}|a_T|^p\lesssim \int|\tilde{f}_\theta|^p\sum_{T\in\H_\theta}\phi_T\lesssim \int|\tilde{f}_\theta|^p\lesssim \int|{f}_\theta|^p.$$
\end{proof}
\bigskip

\section{A weak-type reformulation of the problem}\label{s4}
\bigskip

In order to be able to use tools from incidence theory, it is convenient to rephrase the inequality in Theorem \ref{thmmm1}  in terms of  level set estimates.
\begin{definition}Fix $\frac{2(n+1)}{n-1}< p<\infty$ and $\alpha>0$. We say that $P(p,\alpha)$ holds if for each $\lambda>0$, each $0<\delta<1$ and for each  $f$ with Fourier support in $\A_{\delta}$ and with $\|f\|_{\infty,\delta}\le 1$ we have
\begin{equation}
\label{EE3}
|\{|f|>\lambda\}|\lesssim (\frac{\delta^{-\frac{n-1}{4}+\frac{n+1}{2p}-\frac{\alpha}{p}}}{\lambda})^p\|f\|_2^2,
\end{equation}
where the implicit constant depends only on $n,p,\alpha$.
\end{definition}
Note that due to multiplication invariance, $P(p,\alpha)$ is equivalent with asking that
$$|\{|f|>\lambda\}|\lesssim (\frac{\delta^{-\frac{n-1}{4}+\frac{n+1}{2p}-\frac{\alpha}{p}}}{\lambda})^p\|f\|_2^2\|f\|_{\infty,\delta}^{p-2}$$
for each  $f$ with Fourier support in $\A_{\delta}$. We next prove that $P(p,\alpha)$ undergoes a self-improvement, in that the quantity $\|f\|_2^2\|f\|_{\infty,\delta}^{p-2}$ on the right can be replaced by the smaller $\|f\|_{p,\delta}^p$. The two mechanisms that allow for this improvement are the decomposition using balanced $N$-functions and the equivalence \eqref{EE10} of these quantities in this case.

\begin{proposition}
\label{equivproppalpha}
Assume $P(p,\alpha)$ holds for some $p,\alpha$.
Then we also have
\begin{equation}
\label{EE4}
\|f\|_p\lessapprox\delta^{-\frac{n-1}{4}+\frac{n+1}{2p}-\frac{\alpha}{p}}\|f\|_{p,\delta}
\end{equation}
for each  $f$ with Fourier support in $\A_{\delta}$.
\end{proposition}
\begin{proof}
Fix $f$ with Fourier support in $\A_{\delta}$. We first note that the hypothesis implies that
\begin{equation}
\label{EE6}
\|f\|_p\lessapprox\delta^{-\frac{(n-1)}{4}+\frac{n+1}{2p}-\frac{\alpha}p}\|f\|_{2}^{\frac2p}\|f\|_{\infty,\delta}^{1-\frac2p}.
\end{equation}
To see this it suffices to assume $\|f\|_{\infty,\delta}=1$.
Since
\begin{equation}
\label{EE16}
\|f\|_\infty\le \delta^{-\frac{n-1}{4}}\|f\|_{\infty,\delta}
\end{equation}
we can write
$$|f|\approx \sum_{\lambda\text{ dyadic: }\lambda \lesssim \delta^{-\frac{n-1}{4}}}\lambda1_{|f|\approx \lambda}.$$
Note that by the triangle inequality and Tchebyshev's inequality
$$\|\sum_{\lambda\text{ dyadic: }\lambda \lesssim \delta^{C}}\lambda1_{|f|\approx \lambda}\|_p\le \sum_{\lambda\text{ dyadic: }\lambda \lesssim \delta^{C}}\lambda^{1-\frac2p}\|f\|_{2}^{\frac2p}\le \delta^{-\frac{(n-1)}{4}+\frac{n+1}{2p}-\frac{\alpha}p}\|f\|_{2}^{\frac2p}\|f\|_{\infty,\delta}^{1-\frac2p},$$
if $C$ is chosen large enough.

On the other hand, $P(p,\alpha)$ and the triangle inequality imply that
$$\|\sum_{\lambda \text{ dyadic}\atop{\delta^{C}\lesssim \lambda \lesssim \delta^{-\frac{n-1}{4}}}}\lambda1_{|f|\approx \lambda}\|_p\le \sum_{\lambda \text{ dyadic}\atop{\delta^{C}\lesssim \lambda \lesssim \delta^{-\frac{n-1}{4}}}}\delta^{-\frac{(n-1)}{4}+\frac{n+1}{2p}-\frac{\alpha}p}\|f\|_{2}^{\frac2p}\lessapprox$$
$$\lessapprox \delta^{-\frac{(n-1)}{4}+\frac{n+1}{2p}-\frac{\alpha}p}\|f\|_{2}^{\frac2p},$$
and \eqref{EE6} follows.

We next argue that to prove \eqref{EE4} it suffices to prove the localized version
\begin{equation}
\label{EE7}
\|f\|_{L^p(Q)}\lessapprox\delta^{-\frac{n-1}{4}+\frac{n+1}{2p}-\frac{\alpha}{p}}\|f\|_{p,2\delta}
\end{equation}
for each $\frac12\delta^{-1}$-cube $Q$ and each $f$ Fourier supported in $\A_{2\delta}$. Indeed, then by invoking the rapid decay of $\psi$ we get for each  $f$ Fourier supported in $\A_{\delta}$
$$\|f\|_{L^p(\R^n)}^p=\sum_{Q'}\|\sum_Q\psi_Q^2f\|_{L^p(Q')}^p\lesssim \sum_{Q,Q'}\|\psi_Qf\|_{L^p(Q')}^p(1+\frac{\dist(Q,Q')}{\delta^{-1}})^{-C}.$$
Apply now \eqref{EE7} to $\psi_Qf$ to further bound this by
$$\lessapprox\delta^{-\frac{p(n-1)}{4}+\frac{n+1}{2}-{\alpha}}\sum_Q\|\psi_Qf\|_{p,2\delta}^p.$$
It is easy to see that
$$\|\psi_Qf\|_{p,2\delta}\lesssim (\sum_{\theta\in \P_\delta}\|\psi_Qf_\theta\|_p^2)^{1/2},$$
Using Minkowski's inequality we conclude that
$$\|f\|_{L^p(\R^n)}^p\lessapprox \delta^{-\frac{p(n-1)}{4}+\frac{n+1}{2}-{\alpha}}(\sum_\theta\|\sum_Q\psi_Qf_\theta\|_p^2)^{p/2}\lesssim \delta^{-\frac{p(n-1)}{4}+\frac{n+1}{2}-{\alpha}}\|f\|_{p,\delta}^p.$$

Finally, we prove \eqref{EE7}.  Assume $\|f\|_{p,2\delta}=1$. Using this and Bernstein's inequality for each $\theta\in\P_{2\delta}$ we get that $\|f\|_{\infty,2\delta}\lesssim \delta^{-C}$, for some large enough $C$ whose value is not important.

Write now using Lemma \ref{WPD}
$$f=\sum_{\lambda\lesssim \delta^{-C}}\lambda f_\lambda.$$
As $\|f_\lambda\|_\infty\lesssim 1$, for $C'$ large enough we have
$$\|\sum_{\lambda\lesssim \delta^{C'}}\lambda f_\lambda\|_{L^p(Q)}\lesssim \delta^{-\frac{n-1}{4}+\frac{n+1}{2p}-\frac{\alpha}{p}},$$
so \eqref{EE7} is satisfied for this part of $f$. Since there are $\lessapprox 1$ dyadic $\lambda$ in $[\delta^{C'},\delta^{-C}]$, it remains to prove that for each such $\lambda$
$$\|\lambda f_\lambda\|_{L^p(\R^n)}\lessapprox\delta^{-\frac{n-1}{4}+\frac{n+1}{2p}-\frac{\alpha}{p}}.$$
Use \eqref{EE8} and Lemma \ref{lemmabalancednfunc} $(i)$ to argue that each such $f_\lambda$ is the sum of $\lessapprox 1$ balanced $N$-functions $f_{\lambda,i}$, where $N=\frac1{2\delta}$. It thus suffices to prove that for each $i$
\begin{equation}
\label{dchgchdyrerefuyregyifudtri7t}
\|\lambda f_{\lambda,i}\|_{L^p(\R^n)}\lesssim\delta^{-\frac{n-1}{4}+\frac{n+1}{2p}-\frac{\alpha}{p}}.
\end{equation}
Note that by parabolic rescaling like in Section \ref{s5}, this time with a factor of $1/2$, we can assume that $\supp(\widehat{f_{\lambda,i}})\subset \A_{\delta}$.
Then \eqref{dchgchdyrerefuyregyifudtri7t} follows by combining \eqref{EE6}, \eqref{EE10}, \eqref{EE2} and \eqref{EE8}.
\end{proof}
\bigskip

In the remaining part of the paper we will verify that $P(p,\alpha)$ holds for each $p>\frac{2(n+2)}{n-1}$ and each $\alpha>0$. Proposition \ref{equivproppalpha} will then immediately imply Theorem \ref{thmmm1}. The main tool  is induction on scales.
\begin{proposition}
\label{mainmech}
Fix $p>\frac{2(n+2)}{n-1}$ and suppose that $P(p,\alpha)$ holds for some $\alpha$. Then $P(p,\gamma\alpha)$ also holds, where $0<\gamma<1$ is independent of $\alpha$.
\end{proposition}
We postpone the proof of the proposition to the next sections. For now, let us see how this proposition implies that $P(p,\alpha)$ holds for each $p>\frac{2(n+2)}{n-1}$ and each $\alpha>0$. It suffices to check that $P(p,\alpha)$ holds for large enough $\alpha$.

Fix $f$ with Fourier support in $\A_{\delta}$ and with $\|f\|_{\infty,\delta}\le 1$. Note first that $\{|f|>\lambda\}=\emptyset$ if $\lambda>\delta^{-\frac{n-1}{4}}$, due to \eqref{EE16}. On the other hand,
$$|\{|f|>\lambda\}|\le \lambda^{-2}\|f\|_2^2.$$
It now suffices to note that for $\alpha$ large enough (depending on $p,n$) we have
$$\lambda^{-2}\le (\frac{\delta^{-\frac{n-1}{4}+\frac{n+1}{2p}-\frac{\alpha}{p}}}{\lambda})^p$$
for each
$\lambda\le\delta^{-\frac{n-1}{4}}$.
\bigskip

\section{Parabolic rescaling}\label{s5}
\bigskip

Fix $\frac12\ge \sigma\ge \delta=\frac1N$. Let $Q$ be a $\sigma N$-cube and assume $\supp(\hat{f})\subset \A_\delta$. Define the operator $T_Qf=\psi_Qf$. Then recall that  $\supp(\widehat{T_Qf})\subset\A_{\frac\delta\sigma}$. The next lemma essentially shows the relation between $\|T_Qf\|_{p,\frac\delta\sigma}$ and $\|f\|_{p,\delta}$.
\begin{lemma}
We have for each $2\le p< \infty$ and $\sigma\gtrsim \delta^{1/2}$
\begin{equation}
\label{EE25}
\|T_Qf\|_{p,\frac\delta\sigma}\lessapprox \sigma^{\frac{n+1}{2p}-\frac{n-1}{4}}\|f\|_{p,\delta}.
\end{equation}
Moreover, if $\sigma=\delta^{1/2}$
\begin{equation}
\label{EE13}
\|T_Qf\|_{2}\lesssim N^{n/4}\|f\|_{\infty,\delta}.
\end{equation}
\end{lemma}
\begin{proof}An argument based on the wave packet decomposition similar to the one from the last part of the proof of Proposition \ref{equivproppalpha} shows that \eqref{EE25} will follow if we prove
\begin{equation}
\label{EE12}
\|T_Qf\|_{p,\frac\delta\sigma}\lesssim \sigma^{\frac{n+1}{2p}-\frac{n-1}{4}}\|f\|_2^{\frac2p}\|f\|_{\infty,\delta}^{1-\frac2p}.
\end{equation}
It of course suffices to prove  \eqref{EE12} for $p=2$ and $p=\infty$. When $p=2$, this amounts to proving that
$$\|T_Qf\|_2\lesssim \sigma^{1/2}\|f\|_2.$$
Due to orthogonality, it will suffice to prove this if $\hat{f}$ is supported in a $\approx \delta^{1/2}\times\ldots\times \delta^{1/2}\times \delta$ rectangular box
containing a given $\theta\in \P_\delta$. Due to the  essential rotation invariance of $\psi_Q$, we may as well assume this box is $[0,C\delta^{1/2}]^{n-1}\times [0,C\delta]$. To evaluate $\|\widehat{T_Qf}\|_2$, note that convolution with the $L^1$ normalized kernel $\widehat{\psi_Q}$ does not affect the $L^2$ norm on the first $n-1$ coordinates, as the scale $(\sigma N)^{-1}$ of the convolution kernel is smaller than the scale $\delta^{1/2}$ of the support of $\hat{f}$. However, on the $n^{th}$ coordinate, the scale of the convolution kernel is $\sigma^{-1}$ times larger than that of $\hat{f}$, which accounts for the final $\sigma^{1/2}$ gain on the $L^2$ norm of $T_Qf$.

The bound for $p=\infty$ will follow since each sector $\alpha\in \P_{\frac\delta\sigma}$ intersects $\lesssim \sigma^{\frac{n-1}{2}}$ sectors $\theta\in \P_\delta.$

Finally, \eqref{EE13} will follow by noting that $T_Qf_\theta$ and $T_Qf_{\theta'}$ are orthogonal if $\theta,\theta'$ are not sufficiently close neighbors.
\end{proof}
\bigskip

We get in particular that
$$\|T_\Delta f\|_{p,\delta^{1/2}}^p\lessapprox\delta^{-\frac{p(n-1)}{8}+\frac{n+1}{4}}\|f\|_{p,\delta}^p$$
for each $N^{1/2}$-cube $\Delta$. We will need a version of this inequality where we sum over all $N^{1/2}$-cubes $\Delta$. To achieve this we will exploit the spatial almost orthogonality of the functions $T_\Delta f$ and parabolic rescaling.
\begin{proposition}
\label{dfyyer7guiouy8ugiogbhjgfvlkfgnbjkl}
Let $\supp(\hat{f})\subset\A_\delta$ and assume $P(p,\alpha)$ holds for some $p,\alpha$. Then
$$\sum_{\Delta\;\;N^{1/2}\text{-cube}}\|T_\Delta f\|_{p,\delta^{1/2}}^p\lessapprox\delta^{-\frac{p(n-1)}{8}+\frac{n+1}{4}-\frac{\alpha}{2}}\|f\|_{p,\delta}^p.$$
\end{proposition}
\begin{proof}
In the first step we get rid of the spatial cutoffs $\psi_\Delta$. By definition
$$\sum_{\Delta\;\;N^{1/2}\text{-cube}}\|T_\Delta f\|_{p,\delta^{1/2}}^p=\sum_{\Delta\;\;N^{1/2}\text{-cube}}(\sum_{\tau\in \P_{\delta^{1/2}}}\|(T_\Delta f)_\tau\|_p^2)^{p/2}.$$
Write as before using a partition of unity
$$f=\sum_{\tau'\in \P_{\delta^{1/2}}}\tilde{f}_{\tau'}$$
with $f_{\tau'}$ Fourier supported in $\tau'$.
The key observation is that
$$(T_\Delta f)_\tau=\sum_{\tau'}(T_\Delta \tilde{f}_{\tau'})_\tau,$$
where the sum on the right runs over the $O(1)$ neighbors of $\tau$. This is immediate due to frequency support considerations. We conclude that
$$\sum_{\Delta\;\;N^{1/2}\text{-cube}}(\sum_{\tau\in \P_{\delta^{1/2}}}\|(T_\Delta f)_\tau\|_p^2)^{p/2}\lesssim \sum_{\Delta\;\;N^{1/2}\text{-cube}}(\sum_{\tau\in \P_{\delta^{1/2}}}\|\psi_\Delta \tilde{f}_\tau\|_p^2)^{p/2}.$$
We further bound this first by using Minkowski's inequality, then by using that $$\|\sum_\Delta|\psi_\Delta|^p\|_\infty\lesssim 1$$ to get
$$\le (\sum_{\tau\in \P_{\delta^{1/2}}}(\sum_{\Delta\;\;N^{1/2}\text{-cube}}\|\psi_\Delta \tilde{f}_\tau\|_p^p)^{2/p})^{p/2}\lesssim (\sum_{\tau\in \P_{\delta^{1/2}}}\|f_\tau\|_p^2)^{p/2}.$$
\bigskip

 In the second step it remains to prove that for each $\tau\in \P_{\delta^{1/2}}$
$$\|f_\tau\|_p^p\lessapprox (\sum_{\theta\in \P_\delta:\:\theta\cap \tau\not=\emptyset}\|f_\theta\|_p^2)^{p/2}\delta^{-\frac{p(n-1)}{8}+\frac{n+1}{4}-\frac{\alpha}{2}}.$$
Let $a=(a_1,\ldots,a_{n-1})$ be the center of the  $\delta^{1/4}$-cube $C_\tau$, see \eqref{EE14}. We will perform the parabolic rescaling via the linear transformation
$$L_\tau(\xi_1,\ldots,\xi_n)=(\xi_1',\ldots,\xi_n')=(\frac{\xi_1-a_1}{\delta^{1/4}},\ldots,\frac{\xi_{n-1}-a_{n-1}}{\delta^{1/4}}, \frac{\xi_n-2\sum_{i=1}^{n-1}a_i\xi_i+\sum_{i=1}^{n-1}a_i^2}{\delta^{1/2}}).$$
Note that
$$\xi_n'-\sum_{i=1}^{n-1}{\xi_i'}^2={\delta^{-1/2}}(\xi_n-\sum_{i=1}^{n-1}\xi_i^2).$$
It follows that $L_\tau$ maps the Fourier support $\A_\delta\cap \tau$ of $f_\tau$ to $\A_{\delta^{1/2}}\cap ([-\frac12,\frac12]^{n-1}\times\R)$. Also, for each $\tau'\in \P_{\delta^{1/2}}$ we have that $L_\tau(\theta)=\tau'$ for some $\theta\in \P_\delta$ with $\theta\cap \tau\not =\emptyset$. Thus
$$\|f_\tau\|_p^p=\|g\|_p^p(\operatorname{det}(L_\tau))^{1-p},$$
where $g$ is the $L_\tau$ dilation of $f_\tau$ Fourier supported in $\A_{\delta^{1/2}}\cap ([-\frac12,\frac12]^{n-1}\times\R)$.
By invoking \eqref{EE4} at scale $\delta^{1/2}$ we get
$$\|g\|_p^p\lessapprox\delta^{-\frac{p(n-1)}{8}+\frac{n+1}{4}-\frac{\alpha}{2}}(\sum_{\tau'\in \P_{\delta^{1/2}}}\|g_{\tau'}\|_p^2)^{p/2}.$$
We are done if we use  the fact that
$$\|f_{\theta}\|_p^p=\|g_{\tau'}\|_p^p(\operatorname{det}(L_\tau))^{1-p}$$
whenever $L_\tau(\theta)=\tau'$.
\end{proof}

\bigskip

\section{Localization via incidence theory}\label{s6}
\bigskip

 The following is a rescaled version of Lemma 4.4 in \cite{LW}.

\begin{proposition}
\label{incwolffplandp}
Let $W\subset \R^n$ be a measurable set and let $\H$ be a collection of separated $N$-tubes. Let $\epsilon<\frac12$.  Then there exists a relation $\sim$ between $N^{1-\epsilon}$-cubes $Q$ in $\R^n$ and the tubes in $\H$ and there exists a constant $c_n$ depending only on $n$  such that
\begin{equation}
\label{EE18}
|\{Q:\;T\sim Q\}|\lessapprox 1\text{ for all }\;T\in\H
\end{equation}
\begin{equation}
\label{EE19}
\int_W\sum_{T\in\H:T\not\sim Q(x)}\phi_T(x)dx\lessapprox N^{\epsilon c_n}|\H|^{1/2}|W|,
\end{equation}
where (ignoring zero measure sets) $Q(x)$ denotes the unique $N^{1-\epsilon}$-cube containing $x$.
\end{proposition}
The left hand side essentially counts the incidences between tubes and points. We refer the reader to \cite{LW} for the proof, which is remarkably simple. The key tool is the axiom that two points determine a line. We describe the proof in the case when we count incidences between finitely many points $\P$ and lines $\W$ in $\R^n$. We declare that $L\sim P$ if $P$ is the only point in $L\cap \P$. Note that
$$|\{P\in\P:\;L\sim P\}|\le 1\text{ for all }\;L\in\W,$$
in analogy to \eqref{EE18}.

Also, by Cauchy-Schwartz
$$I:=\sum_{P\in\P}\sum_{L\in\W\atop{L\not\sim P}}1=\sum_{L\in\W'}\sum_{P\in\P\cap L}1\le {|\W'|}^{1/2}(\frac{|\P^2|}{2}+I)^{1/2},$$
where $\W'$ denotes the lines in $\W$ containing more than one point from $\P$. Since $I\ge 2|\W'|$, we must have
$$I\le |\W'|^{1/2}|\P|\le |\W|^{1/2}|\P|,$$
which is the analogue of \eqref{EE19}.

There is no obvious way in which Proposition \ref{incwolffplandp} may be improved without any assumption on the structure of $W$. The generic example -relevant also for our forthcoming application- is when $W$ consists of $J$ pairwise disjoint $N^{1/2}$-cubes. Even in two dimensions, one can easily device such a $W$ and a collection of tubes $\H$ such that the number of incidences between the cubes and tubes is $\gg (J|\H|)^{2/3}+|\H|+J$, in sharp contrast with the Szemer\'edi-Trotter theorem. The explanation is that, without further assumption on $W$ or $\H$, the second crucial line-point axiom -the one about two lines intersecting in at most one point- fails dramatically in the tube-cube world. This axiom is regained if the tubes are transverse, a feature incorporated in a deep way by the multilinear theory. And indeed, the range we obtain in Theorem \ref{thmmm1} will follow by combining Wolff's technology with the more recent multilinear machinery, via Theorem \ref{thm:Bor}.
\bigskip

Here is how Proposition \ref{incwolffplandp} contributes to the general argument.
For $p>\frac{2(n+2)}{n-1}$ we fix a small enough $\epsilon_p>0$ whose value will be clear from the argument in the following sections.
\begin{definition}
Let $f$ be an $N$-function. We say that it $p$- localizes at some $\lambda>0$ if for each $N^{1-\epsilon_p}$-cube $Q$ there are subfunctions $f^Q$ of $f$ such that
\begin{equation}
\label{EE21}
|\{|f|\ge \lambda\}|\lesssim \sum_Q|Q\cap \{|f^Q|\ge \lambda/2\}|,
\end{equation}
and for each $\theta\in\P_{1/N}$ we have
\begin{equation}
\label{EE20}
\sum_Q|\H(f^Q,\theta)|\lessapprox |\H(f,\theta)|.
\end{equation}
\end{definition}
The reader will note that the concept of localization is sensitive to the value of $p$.
\begin{lemma}
\label{hfuiy43578tyufrugf}
Let $f$ be an $N$-function and let $\lambda>0$ be such that
$$
|\H(f)|\le C^{-1}N^{-4\epsilon_pc_n}\lambda^2,
$$
for some large enough $C$ depending on $n$.
Then $f$ $p$- localizes at $\lambda$.
\end{lemma}
\begin{proof}
Let $W=\{|f|\ge \lambda\}$.
Using Proposition \ref{incwolffplandp} we get
$$\int_W\sum_{T\in\H(f):T\not\sim Q(x)}\phi_T(x)dx\lessapprox C^{-1}N^{-\epsilon_p c_n}\lambda|W|\le \frac{\lambda}4|W|,$$
if $C$ is chosen large enough.
Hence there is a set $W^*\subset W$ with $|W|\lesssim |W^*|$ such that
$$\sum_{T\in\H(f):T\not\sim Q(x)}\phi_T(x)\le \frac\lambda2,\;\text{ for }x\in W^*.$$
For each $Q$ define $f^Q=\sum_{T\in \H(f):T\sim Q}f_T$. Note that if $x\in W^*\cap Q$ we have
$$|f(x)-f^Q(x)|=|\sum_{T\in \H(f):T\not\sim Q}f_T(x)|\le \sum_{T\in\H(f):T\not\sim Q(x)}\phi_T(x)\le \frac\lambda2,$$
and \eqref{EE21} follows by summing over $Q$.

Finally, note that \eqref{EE20} is an immediate consequence of \eqref{EE18}.
\end{proof}
We are now able to verify Proposition \ref{mainmech} in the particular case when we deal with $N$-functions that localize.
\begin{proposition}
\label{dbncvyer78f589ty}
Assume $P(p,\alpha)$ holds for some $p>\frac{2(n+2)}{n-1}$ and $\alpha>0$. Let $f$ be an $N$-function which $p$-  localizes at some  $\lambda$.  Then $P(p,\beta)$ will hold for $f$ and $\lambda$ with $\beta:=\alpha(1-\frac{\epsilon_p}{2})$, more precisely
$$|\{|f|\ge \lambda\}|\lesssim(\frac{\delta^{-\frac{n-1}{4}+\frac{n+1}{2p}-\frac{\beta}{p}}}{\lambda})^p\|f\|_{p,\delta}^p.$$
\label{maingain}
\end{proposition}
\begin{proof}
In the following, $Q$ will refer to $N^{1-\epsilon_p}$- cubes. We have due to localization
$$|\{|f|\ge \lambda\}|\lesssim\sum _Q|Q\cap\{|f^Q|\ge \lambda/2\}|\le \sum _Q|\{|T_Qf^Q|\gtrsim \lambda\}|.$$
Invoking $P(p,\alpha)$ and Proposition \ref{equivproppalpha} for $T_Qf^Q$ we can bound the above by
$$\lessapprox(\frac{\delta^{(-\frac{n-1}{4}+\frac{n+1}{2p}-\frac{\alpha}{p})(1-\epsilon_p)}}{\lambda})^p\sum _Q\|T_Qf^Q\|_{p,\delta^{1-\epsilon_p}}^p,$$
while \eqref {EE25} with $\sigma=\delta^{\epsilon_p}$ establishes the bound
$$\lessapprox(\frac{\delta^{(-\frac{n-1}{4}+\frac{n+1}{2p}-\frac{\alpha}{p})(1-\epsilon_p)}}{\lambda})^p\delta^{\frac{\epsilon_p(n+1)}{2}-\frac{\epsilon_p p(n-1)}{4}}\sum _Q\|f^Q\|_{p,\delta}^p.$$
Finally note that by \eqref{EE2}, Minkowski's inequality and \eqref{EE20}
$$\sum _Q\|f^Q\|_{p,\delta}^p
\sim \sum _Q N^{\frac{n+1}{2}}(\sum_{\theta}|\H(f^Q,\theta)|^{\frac2p})^{p/2}\le$$
$$\le N^{\frac{n+1}{2}}(\sum_\theta(\sum_Q|\H(f^Q,\theta)|)^{2/p})^{p/2}\lessapprox N^{\frac{n+1}{2}}(\sum_\theta|\H(f,\theta)|^{2/p})^{p/2}\sim \|f\|_{p,\delta}^p.$$
The three logarithmic losses due to $\lessapprox$ are compensated by the fact that we choose $\beta=\alpha(1-\frac{\epsilon_p}{2})$ rather than  $\beta=\alpha(1-{\epsilon_p})$
\end{proof}
\bigskip

\section{Proof of Proposition \ref{mainmech}}\label{s7}
\bigskip

We will show that we can take $\gamma=1-\frac{\epsilon_p}{8}$, if  $\epsilon_p$ is small enough.

Let $\delta=\frac1N$ and assume $\|f\|_{\infty,\delta}=1$.  Observe first that  \eqref{EE3} holds for each $\alpha>0$,  if $\lambda\lesssim N^{\frac{n-1}{4}-\frac{n-1}{2(pn-p-2n)}}$. Indeed, using the subcritical estimate \eqref{Snew29} and \eqref{EE1} we get
$$\|f\|_{\frac{2n}{n-1}}\lesssim_\epsilon{N^{\epsilon}}\|f\|_{\frac{2n}{n-1},\delta}\le {N^{\epsilon}}\|f\|_2^{\frac{n-1}{n}}.$$
Thus by Tchebyshev's inequality, for each $\lambda,\epsilon>0$
$$|\{|f|>\lambda\}|\lesssim_\epsilon N^{\epsilon}\lambda^{-\frac{2n}{n-1}}\|f\|_2^2.$$
It suffices now to note that for each  $\lambda\lesssim N^{\frac{n-1}{4}-\frac{n-1}{2(pn-p-2n)}}$ we have
$$\lambda^{-\frac{2n}{n-1}}N^{\epsilon}\le (\frac{\delta^{-\frac{n-1}{4}+\frac{n+1}{2p}-\frac{\alpha}{p}}}{\lambda})^p,$$
if $\epsilon$ is small enough.

We can now assume $\lambda\gtrsim N^{\frac{n-1}{4}-\frac{n-1}{2(pn-p-2n)}}$. First note that for each $\epsilon>0$
$$|\{|f|>\lambda\}|=|\{|\sum_{\Delta\;N^{1/2}-\text{cube}}T_\Delta f|>\lambda\}|\le\sum_{\Delta\;N^{1/2}-\text{cube}}|\{|T_\Delta f|\gtrsim_{\epsilon}\delta^{\epsilon}\lambda\}|,$$
due to the Schwartz decay of $\psi$ and the fact that $\|f\|_\infty\le N^{\frac{n-1}{4}}$.

By Lemma \ref{WPD}
$$T_\Delta f=\sum_{h\lesssim \|T_\Delta f\|_{\infty, \delta^{1/2}}}hf_{\Delta,h}.$$
We can further bound the above by
$$\sum_{\Delta\;N^{1/2}-\text{cube}}|\{|\sum_{\delta^{\epsilon}\lambda\lesssim_\epsilon h\lesssim \|T_\Delta f\|_{\infty, \delta^{1/2}}}hf_{\Delta,h}|\gtrsim_{\epsilon}\delta^{\epsilon}\lambda\}|\le$$
\begin{equation}
\label{EE41}
\sum_{\Delta\;N^{1/2}-\text{cube}}\sum_{\delta^{\epsilon}\lambda\lesssim_\epsilon h\lesssim \|T_\Delta f\|_{\infty, \delta^{1/2}}}|\{|f_{\Delta,h}|\gtrsim_{\epsilon}h^{-1}\delta^{2\epsilon}\lambda\}|,
\end{equation}
since there are $\lessapprox 1$ many $h$ in the summation.
The crucial gain will come from the fact that $f_{\Delta,h}$ $p$- localizes at  $h^{-1}\delta^{\epsilon_p'}\lambda$, if $\epsilon_p'$ is chosen small enough. Let us see why this holds.
From \eqref {EE8} with $p=2$ and \eqref{EE13} we obtain
$$h^2N^{\frac{n+1}{4}}|\H(f_{\Delta,h})|\lesssim \|T_\Delta f\|_{2,\delta^{1/2}}^2\lesssim N^{n/2}.$$
It is now an easy computation to check that the hypothesis of Lemma \ref{hfuiy43578tyufrugf} is satisfied with $f$ and $\lambda$ replaced by $f_{\Delta,h}$ and $ h^{-1}\delta^{\epsilon_p'}\lambda,$ if we select a small enough $\epsilon_p'$ and $\epsilon_p$.

With localization at hand, we use Proposition \ref{dbncvyer78f589ty}, then \eqref{EE8} and Proposition \ref{dfyyer7guiouy8ugiogbhjgfvlkfgnbjkl} to bound  \eqref{EE41} when $\epsilon=\frac{\epsilon_p'}{2}$ by
$$\sum_{\Delta\;N^{1/2}-\text{cube}}\sum_{\delta^{\frac{\epsilon_p'}2}\lambda\lesssim h\lesssim \|T_\Delta f\|_{\infty, \delta^{1/2}}}(h\frac{\delta^{-\frac{n-1}{8}+\frac{n+1}{4p}-\frac{\beta}{2p}-\epsilon_p'}}{\lambda})^p\|f_{\Delta,h}\|_{p,\delta^{1/2}}^p\lessapprox$$
$$\sum_{\Delta\;N^{1/2}-\text{cube}}(\frac{\delta^{-\frac{n-1}{8}+\frac{n+1}{4p}-\frac{\beta}{2p}-\epsilon_p'}}{\lambda})^p\|T_\Delta f\|_{p,\delta^{1/2}}^p\lessapprox$$
$$\delta^{-\frac{p(n-1)}{8}+\frac{n+1}{4}-\frac{\alpha}{2}}(\frac{\delta^{-\frac{n-1}{8}+\frac{n+1}{4p}-\frac{\beta}{2p}-\epsilon_p'}}{\lambda})^p\|f\|_{p,\delta}^p=
(\frac{\delta^{-\frac{n-1}{4}+\frac{n+1}{2p}-\frac{\beta'}{p}}}{\lambda})^p\|f\|_{p,\delta}^p,$$
where $$\beta'={\alpha(1-\frac{\epsilon_p}{4})+p\epsilon_p'}.$$
Choosing $\epsilon_p'$ small enough  will guarantee that $\beta'<\alpha(1-\frac{\epsilon_p}{8})$. This ends the proof.
\bigskip

\section{Additive energies and incidence theory}
\label{s8}
\bigskip

Given an integer $k\ge 2$ and a set $\Lambda$ in $\R^n$ we introduce its $k$- energy
$$\E(\Lambda)=|\{(\lambda_1,\ldots,\lambda_{2k})\in \Lambda^{2k}:\;\lambda_1+\ldots+\lambda_k=\lambda_{k+1}+\ldots\lambda_{2k}\}|.$$

We recall the point-line incidence theorem due to Szemer\'edi and Trotter
\begin{theorem}[\cite{ST}]
 There are $O(|\W|+|\P|+(|\W||\P|)^{2/3})$ incidences between any collections $\W$ and $\P$ of  lines and points in the plane.
\end{theorem}
Up to extra logarithmic factors, the same thing is conjectured to hold if lines are replaced with circles. Another related conjecture is
\begin{conjecture}[The unit distance conjecture]
The number of unit distances between $N$ points in the plane is always $\lesssim_\epsilon N^{1+\epsilon}$
\end{conjecture}

The point-circle and the unit distance conjectures are thought to be rather hard, and only partial results are known.

We now offer a first glimpse into Conjecture \ref{c3} by proving Theorem \ref{thmmmm4}. This will expose the role played by incidence theory in the discrete restriction phenomenon.

Recall $n=3$ and $\Lambda\subset P^2 $. The following parameter encodes the "additive geometry" of $\Lambda$
$$\upsilon:=\min\left\{|\eta_1+\eta_2-\eta_3-\eta_4|:\eta_i\in \Lambda\text{ and }|\eta_1+\eta_2-\eta_3-\eta_4|\not=0\right\}.$$
We show that Theorem \ref{thmmmm4} holds if $R\gtrsim \frac{|\Lambda|^2}{\upsilon}$. Fix such an $R$.

Using restricted type interpolation  it suffices to prove
$$\frac1{|B_R|}\int_{B_R}|\sum_{\eta\in\Lambda'} e(x\cdot\eta)|^{4}dx\lesssim_\epsilon |\Lambda'|^{2+\epsilon} ,$$
for each subset $\Lambda'\subset\Lambda$. See Section 6 in \cite{BD2} for details on this type of approach.

Expanding the $L^4$ norm we need to prove
$$|\sum_{\eta_i\in\Lambda'}\frac1{R^3}\int_{B_R}e((\eta_1+\eta_2-\eta_3-\eta_4)\cdot x)dx|\lesssim_\epsilon |\Lambda'|^{2+\epsilon} .$$
Note that if $A\not=0$
$$|\int_{-R}^{R}e(At)dt|\le A^{-1}.$$
Using this we get that
$$|\sum_{\eta_i\in\Lambda'\atop{|\eta_1+\eta_2-\eta_3-\eta_4|\not=0}}\frac1{R^3}\int_{B_R}e((\eta_1+\eta_2-\eta_3-\eta_4)\cdot x)dx|\le \frac{|\Lambda'|^4}{R\upsilon} \le |\Lambda'|^2.$$

Thus it suffices to prove the following estimate  for the additive energy
\begin{equation}
\label{e1}
\E_2(\Lambda')\lesssim_{\epsilon} |\Lambda'|^{2+\epsilon}.
\end{equation}
Assume
\begin{equation}
\label{EEE1}
\eta_1+\eta_2=\eta_3+\eta_4,
\end{equation}
with $\eta_i:=(\alpha_i,\beta_i,\alpha_i^2+\beta_i^2)$.
It has been observed in \cite{Bo3} that given $A,B,C\in\R$, the equality
$$\eta_1+\eta_2=(A,B,C)$$ implies that for $i\in\{1,2\}$
\begin{equation}
\label{EEE2}
(\alpha_i-\frac{A}{2})^2+(\beta_i-\frac{B}2)^2=\frac{2C-A^2-B^2}{4}.
\end{equation}
Thus the four points $P_i=(\alpha_i,\beta_i)$ corresponding to any additive quadruple \eqref{EEE1} must belong to a  circle. As observed in \cite{Bo3}, this is enough to conclude \eqref{e1} in the lattice case, as circles of radius $M$ contain $\lesssim_\epsilon M^\epsilon$ lattice points. The bound \eqref{e1} also follows immediately if one assumes the circle-point incidence conjecture.

We need however a new observation. Note that if \eqref{EEE1} holds then in fact both $P_1,P_2$ and $P_3,P_4$ are diametrically opposite on the circle \eqref{EEE2}. Thus each additive quadruple  gives rise to a distinct right angle, the one subtended by $P_1,P_2,P_3$ (say).
The estimate \eqref{e1} is then an immediate consequence of the following application of the Szemer\'edi-Trotter Theorem.
\begin{theorem}[Pach, Sharir, \cite{PS}]
\label{P-Sthm}
The number of repetitions of a given angle among $N$ points in the plane is $O(N^2\log N)$.
\end{theorem}

It has been recognized that the restriction theory for the sphere and the paraboloid are very similar\footnote{A notable difference is the lattice case of the discrete restriction, but that has to do with a rather specialized scenario}. Consequently, one expects not only Theorem \ref{thmmmm4} to be true also for $S^2$, but for a very similar argument to work in that case, too. If that is indeed the case, it does not appear to be obvious. The same argument as above shows that an additive quadruple of points on $S^2$ will belong to a circle on $S^2$,
and moreover the four points will be diametrically opposite in pairs. There will thus be at least $\E_2(\Lambda)$ right angles in $\Lambda$. This is however of no use in this setting, as $\Lambda$ lives in three dimensions. It is proved in \cite{AS} that a set of $N$ points in $\R^3$ has $O(N^{7/3})$ right angles, and moreover this bound is tight in general.

Another idea is to map an additive quadruple to the plane  using the stereographic projection. The resulting four points will again belong to a circle, so the bound on the energy would follow if the circle-point incidence conjecture is proved. Unfortunately, the stereographic  projection does not preserve the property of being diametrically opposite and thus prevents the application of Theorem \ref{P-Sthm}. We thus ask
\begin{question}
Is it true that $\E_2(\Lambda)\lesssim_\epsilon|\Lambda|^{2+\epsilon}$ for each finite $\Lambda\subset S^2$?
\end{question}
One can ask the same question for $P^{n-1}$ and $S^{n-1}$ when $n\ge 4$. At least in the case of $\Lambda\subset P^{n-1}$, it is clear that the best one can hope for (just consider the lattice case, where a lower bound $\E_2(P^{n-1})$ can easily be derived)  is
\begin{equation}
\label{EEE4}
\E_2(\Lambda)\lesssim_\epsilon|\Lambda|^{\frac{3n-5}{n-1}+\epsilon}.
\end{equation}
Interestingly, this follows from the aforementioned result  \cite{AS} when $n=4$, and in fact no $\epsilon$ is needed this time. However, in the same paper \cite{AS} it is proved that this argument fails in  dimensions five and higher: there is a set with $N$ points in $\R^4$ which determines $\gtrsim N^3$ right angles. We point out that Theorem \ref{thmmmm2} (and its analogue for the sphere) implies \eqref{EEE4} for subsets of $P^{n-1}$ (and $S^{n-1}$) when $n\ge 4$, in the case when the points $\Lambda$ are $\sim |\Lambda|^{-1}$ separated.
\bigskip

It is also natural to investigate the two dimensional phenomenon. The approach outlined above for $n=3$  shows that Conjecture \ref{c3} is now equivalent with a positive answer to the following
\begin{question}
Is this true that for each $\delta$- separated $\Lambda\subset S$
$$\E_3(\Lambda)\lesssim_\epsilon\delta^{-\epsilon}|\Lambda|^{3}$$
for either $S=P^1$ or $S=S^1$?
\end{question}
One can be more bold and ask the more general question
\begin{question}
\label{qqqq1}
Is it true that for each $\Lambda\subset S$
$$\E_3(\Lambda)\lesssim_\epsilon|\Lambda|^{3+\epsilon}?$$
\end{question}
Surprisingly, these questions seem to be harder than their three dimensional counterparts solved above. The best that can be said with topological methods seems to be the following
\begin{proposition}For each $\Lambda\subset S$
$$\E_3(\Lambda)\lesssim_\epsilon|\Lambda|^{\frac72+\epsilon}.$$
\end{proposition}
\begin{proof}
This was observed independently by Bombieri-Bourgain \cite{BoBo} when $S=S^1$ and by the author when $S=P^1$. The proofs are very similar, we briefly sketch the details for $S=P^1$. Let $N$ be the cardinality of $\Lambda$.
It goes back to \cite{Bo3} that if
\begin{equation}
\label{EEE5}
(x_1,x_1^2)+(x_2,x_2^2)+(x_3,x_3^2)=(n,j),
\end{equation}
then the point $(3(x_1+x_2),\sqrt{3}(x_1-x_2))$ belongs to the circle centered at $(2n,0)$ and of radius squared equal to $6j-2n^2$.  Note that there are $N^2$ such points with $(x_i,x_i^2)\in \Lambda$, call this set of points $T$. Assume we have $M_n$ such circles containing roughly $2^n$ points  $(3(x_1+x_2),\sqrt{3}(x_1-x_2))\in T$ in such a way that
\eqref{EEE5} is satisfied for some $x_3\in \Lambda$. Then clearly
$$\E_3(S)\lesssim \sum_{2^n\le N}M_n2^{2n}.$$

It is easy to see that
\begin{equation}
\label{EEE6}
M_n2^n\lesssim N^3,
\end{equation}
as each point in $T$ can belong to at most $N$ circles.

The nontrivial estimate is
\begin{equation}
\label{EEE7}
M_n2^{3n}\lesssim N^4,
\end{equation}
which is an immediate consequence of the Szemer\'edi-Trotter Theorem for curves  satisfying the following two fundamental axioms: two curves intersect in $O(1)$ points, and there are $O(1)$ curves passing through any two given points. The number of incidences between such curves and points is the same as in the case of lines and points, see for example Theorem 8.10 in \cite{TV}. Note that since our circles have centers on the $x$ axis, any two points in $T$ sitting in the upper (or lower) half plane determine a unique circle. Combining the two inequalities we get for each $n$
$$M_n2^{2n}\lesssim N^{\frac72}.$$
\end{proof}
\bigskip

In the case when $\Lambda\subset S^1$, the same argument leads to incidences between unit circles and points. The outcome is the same, since for any two points there are two unit circles passing through them. An interesting observation is the fact that Question \ref{qqqq1} has a positive answer if the Unit Distance Conjecture is assumed. Indeed, the argument above presents us with a collection $T$ of $N^2$ points and a collection  of $\lesssim N^3$ unit circles. For $2^n\lesssim N$ let $M_n$ be the number of such circles  with $\sim 2^n$ points. There will be at least $M_n2^n$ unit distances among the $N^2$ points and the $M_n$ centers. The Unit Distances Conjecture forces $M_n2^n\lesssim_\epsilon(M_n+N^2)^{1+\epsilon}$. Since $M_n\lesssim N^3$, it immediately follows that $M_n2^{2n}\lesssim_\epsilon N^{3+\epsilon}$ which gives the desired bound on the energy.

It seems likely that in order to achieve the conjectured bound on $\E_3(\Lambda)$, the structure of $T$ must be exploited, paving the way to algebraic methods. One possibility is to make use of the fact that $T$ has sumset structure.   Another interesting angle for the parabola is the following. Recall that whenever \eqref{EEE5} holds, the three points $(3(x_i+x_j),\sqrt{3}(x_i-x_j))$, $(i,j)\in\{(1,2),(2,3),(3,1)\}$, belong to the circle centered at $(2n,0)$ and of radius squared equal to $6j-2n^2$. One can easily check that if fact they form an equilateral triangle! This potentially opens up the new toolbox of symmetries, since, for example, the rotation by $\pi/3$ about the center of any such circle $C$ will preserve $C\cap T$.
\bigskip

Another interesting question is whether there is a soft(er) argument for Theorem \ref{thmmmm2}, one that does not go through the hard analysis of Theorem \ref{thmmm1}. In particular, it seems tempting to search for  a proof that relies mainly on incidence theory of some of the consequences of Theorem \ref{thmmmm2} when $p$ is an even integer. For example,  when $n=2$ and $p=8$ Theorem \ref{thmmmm2} shows that
$$\E_4(\Lambda)\lesssim_{\epsilon} N^{5+\epsilon},$$
for each $\frac1N$- separated set on $P^1$. An argument like the one in Proposition \ref{qqqq1} relying on two applications of Szemer\'edi-Trotter produces the upper bound $N^{5\frac14+\epsilon}$ without any assumption on separation.

We close by mentioning that there are similar interesting questions for the cone, but we have decided to not explore them here.

\end{document}